\documentclass[10pt,letterpaper]{amsart}
\usepackage[letterpaper,margin=1.3in,footskip=0.70in]{geometry}

\usepackage[dvipsnames]{xcolor}

\usepackage[inline,shortlabels]{enumitem}
\setlist{noitemsep}

\usepackage{amsmath,amssymb,amsthm}
\usepackage{mathtools}
\usepackage{stmaryrd}
\usepackage{textgreek}
\usepackage[retainorgcmds]{IEEEtrantools}
\usepackage{colonequals}
\usepackage{relsize}
\usepackage[bbgreekl]{mathbbol}
\usepackage{mathrsfs}
\usepackage{tikz-cd}
\usepackage{array}

\usepackage{todonotes}

\DeclareSymbolFontAlphabet{\mathbb}{AMSb} 
\DeclareSymbolFontAlphabet{\mathbbl}{bbold}
\usepackage{graphicx}

\makeatletter
\newcommand{\smallbullet}{} 
\DeclareRobustCommand\smallbullet{%
  \mathord{\mathpalette\smallbullet@{0.7}}%
}
\newcommand{\smallbullet@}[2]{%
  \vcenter{\hbox{\scalebox{#2}{$\m@th#1\bullet$}}}%
}
\makeatother
\usepackage[pdfusetitle,colorlinks,unicode]{hyperref}
\hypersetup{citecolor=black,linkcolor=black,urlcolor=black}
\usepackage[capitalise,noabbrev]{cleveref}
\crefformat{enumi}{#2\textup{(#1)}#3}
\crefformat{equation}{\ensuremath{(#2#1#3)}}
\crefmultiformat{equation}{\ensuremath{(#2#1#3)}}{ and~\ensuremath{(#2#1#3)}}{, \ensuremath{(#2#1#3)}}{, and~\ensuremath{(#2#1#3)}}
\numberwithin{figure}{section}
\numberwithin{equation}{section}
\newtheorem{theorem}[figure]{Theorem}
\newtheorem{lemma}[figure]{Lemma}
\newtheorem{corollary}[figure]{Corollary}
\newtheorem{proposition}[figure]{Proposition}

\theoremstyle{definition}
\newtheorem{definition}[figure]{Definition}
\newtheorem{notation}[figure]{Notation}
\theoremstyle{definition}
\newtheorem{remark}[figure]{Remark}
\theoremstyle{definition}
\newtheorem{example}[figure]{Example}
\theoremstyle{definition}
\newtheorem{construction}[figure]{Construction}

\DeclareMathOperator*{\colim}{colim}
\DeclareMathOperator{\crys}{crys}

\DeclareMathOperator{\gr}{gr}

\DeclareSymbolFontAlphabet{\mathbb}{AMSb} 
\DeclareSymbolFontAlphabet{\mathbbl}{bbold}

\DeclareMathOperator{\Spec}{Spec}

\DeclareMathOperator{\conj}{conj}
\DeclareMathOperator{\Hodge}{Hodge}
\DeclareMathOperator{\Fil}{Fil}
\DeclareMathOperator{\dR}{dR}

\newcommand{\LL}{\L}

\newcommand{\w}{\text}

\newcommand{\TR}{\mathrm{TR}}

\newcommand{\TC}{\mathrm{TC}}

\renewcommand{\L}{\mathbb{L}}
\newcommand{\Z}{\mathbf{Z}}
\newcommand{\F}{\mathbf{F}}
\newcommand{\Q}{\mathbf{Q}}

\newcolumntype{C}[1]{>{\centering\arraybackslash}p{#1}}

\mathchardef\mhyphen="2D

\newcommand\restr[2]{{\left.\kern-\nulldelimiterspace#1\vphantom{\big|}\right|_{#2}}}
\newcommand{\suchthat}{\;\ifnum\currentgrouptype=16 \middle\fi\vert\;}
\newcommand{\cosimp}[3]{\xymatrix@1{#1 \ar@<.4ex>[r] \ar@<-.4ex>[r] & {\ }#2 \ar@<0.8ex>[r] \ar[r] \ar@<-.8ex>[r] & {\ } #3  \cdots }} 
\usepackage[all]{xy}

\AtBeginDocument{%
   \def\MR#1{}
}

\title{ $p$-typical curves on $p$-adic Tate twists and de Rham--Witt forms}

\author{Sanath K. Devalapurkar}
\author{Shubhodip Mondal}

\address[Sanath K. Devalapurkar]{}
\email{sdevalapurkar@math.harvard.edu}

\address[Shubhodip Mondal]{}
\email{smondal@math.ubc.ca}

\begin{document}

\begin{abstract}
We show that de Rham--Witt forms are naturally isomorphic to $p$-typical curves on $p$-adic Tate twists, which answers a question of Artin--Mazur from 1977 pursued in the earlier work of Bloch and Kato. We show this by more generally equipping a related result of Hesselholt on topological cyclic homology with the motivic filtrations introduced by Bhatt--Morrow--Scholze.
\end{abstract}
\maketitle
\vspace{-8.5mm}
\tableofcontents

\section{Introduction}
Let $k$ be a perfect field of characteristic $p>0.$ For a smooth proper scheme $X$ over $k,$ Artin and Mazur \cite{AM} constructed certain formal groups $\Phi^r(X)$ associated to $X$, with the property that $p$-typical curves on $\Phi^r(X)$ recover the slope $[0,1)$ part of crystalline cohomology $H^r_{\crys} (X/W(k))_{\Q}$ (see \cite[Cor.~3.3]{AM}). Roughly speaking, this relies on the fact that the Dieudonn\'e module of $p$-typical curves on the formal group $\widehat{\mathbb G}_m$ is $W(k)$ with the usual $F$ and $V$-operators. In \cite[Qn.~(b)]{AM}, the authors raised the question of whether there is a way to recover the slope $[i,i+1)$ part of $H^r_{\crys} (X/W(k))_{\Q}$ via the formalism of $p$-typical curves on certain group valued functors.

This question was answered by Bloch \cite{bloch} under the hypothesis that $p > \dim X$ by studying $p$-typical curves on symbolic part of the higher algebraic $K$-groups, generalizing the role played by ${{\mathbb G}}_m$ when $i=0.$ The possibility of removing some of these assumptions was expressed in  \cite[p.~635, Rmk.~2]{Kato}.

The goal of our paper is to fully settle the above question of Artin and Mazur. Furthermore, we show that instead of using the algebraic $K$-groups (as done by Bloch \cite{bloch}), one may simply use the $p$-adic Tate twists $\Z_p (n)[n]$ from \cite{BMS2}, which, in some sense, generalizes the role played by the $p$-adic completion of $\mathbb G_m$ when $n=1.$ More precisely, we prove the following, which affirmatively answers \cite[Qn.~(b)]{AM} without any restriction on the dimension (see \cref{artinmazurquestionrem}).

\begin{theorem}\label{mainthmintro}
    Let $S$ be a quasisyntomic $\F_p$-algebra. Then for every $n \ge 0,$ we have a natural isomorphism
$$\prod_{I_p}\L W\Omega^{n-1}_S \simeq \mathrm{fib} \left(\varprojlim_k \Z_p(n)(S[t]/t^k)[n] \to \Z_p(n)(S)[n]\right),$$ where $I_p$ denotes the set of positive integers coprime to $p.$
\end{theorem}{}

In the above, $\LL W\Omega^{n-1}_S$ denotes the animated de Rham--Witt forms as in \cref{derhamw}. By analogy with the classical work of Cartier \cite{Cart}, the right hand side may be regarded as curves on the functor $\Z_p (n)[n]$; which is further equipped with Frobenius and Verschiebung operators $F_m, V_m$ for all $m \ge 0$ (see \cref{constrans}). Let $\mathbb{D}(\Z_p(n)[n]_S)\coloneqq \bigcap_{(m,p)=1, m>1} \mathrm{fib}(F_m)$ (see \cref{ptypica}), which may be regarded as $p$-typical curves on the functor $\Z_p(n)[n].$ Note that $\mathbb{D}(\Z_p(n)[n]_S)$ is naturally equipped with the operators $F\coloneqq F_p$ and $V\coloneqq V_p$. As a corollary of \cref{mainthmintro}, one obtains

\begin{corollary}\label{coro}
    Let $S$ be a quasisyntomic $\F_p$-algebra. Then for each $n \ge 0,$ we have a natural isomorphism 
$$ \LL W \Omega_S^{n-1} \simeq \mathbb{D}(\Z_p(n)[n]_S),$$ which is compatible with the $F$ and $V$ defined on both sides.
\end{corollary}{}

\begin{remark}
Let us explain \cref{mainthmintro} and \cref{coro} in the case $n=1.$ For a quasisyntomic $\F_p$-algebra $S,$ one has $\Z_p(1) (S)[1]\simeq R\Gamma_{\mathrm{\acute{e}t}} (S, \mathbb G_m)^{\wedge_p}.$ More or less by definition, it then follows that the right hand side of \cref{mainthmintro} is isomorphic to the ring of big Witt vectors of $S;$ which is isomorphic to $\prod_{I_p}W(S),$ where $W(S)$ denotes the ring of $p$-typical Witt vectors of $S$ (see, e.g., \cite[Prop.~3.6]{bloch}). Thus \cref{coro} in this case says that $W(S) \simeq \mathbb{D}(\mathbf{Z}_p(1)[1]_S).$
\end{remark}{}

\begin{remark}\label{artinmazurquestionrem}
    Note that by the degeneration of the slope spectral sequence as proved in \cite{Luc1}, the slope $[i, i+1)$ part of $H^r_{\crys}(X/W(k))_{\Q}$ identifies with $H^{r-i}(X, W\Omega^i_X )_\mathbb{Q}$. Thus, our formula recovers the slopes desired in \cite[Qn.~(b)]{AM}. Also, \cref{coro} gives a way to reconstruct the de Rham--Witt forms purely from the $p$-adic Tate twists $\Z_p(n).$
\end{remark}{}

Let us now mention the main ideas that go into the proof of \cref{mainthmintro}, which, in particular, uses some of the recent techniques introduced in $p$-adic geometry by Bhatt--Morrow--Scholze. In \cite{Hess}, Hesselholt proved that there is a natural isomorphism 
\begin{equation}\label{hesselholt}
    \TR(S)[1]\simeq  \mathrm{fib} (\varprojlim_k \TC (S[t]/t^k) \to \TC(S)),
\end{equation}where $\TR$ denotes topological restriction homology and $\TC$ denotes topological cyclic homology\footnote{Combine \cite[Thm.~3.1.9]{Hess} and \cite[Thm.~3.1.10]{Hess} to obtain this statement.}. Further, Hesselholt \cite[Thm.~C]{Hess} showed that if $S$ is smooth, then $\pi_* \TR(S,p) \simeq W\Omega^*_S,$ where $\TR(S,p)$ denotes $p$-typical part of $\TR(S).$ In \cite{BMS2}, Bhatt--Morrow--Scholze constructed a ``motivic" filtration $\Fil^* \TC (S)$ on $\TC(S)$ where the graded pieces $\mathrm{gr}^n \TC(S)$ are given by $\Z_p (n)(S)[2n]$ (\textit{cf}.~\cite{Ben, arpon}). Using techniques similar to \cite{BMS2} and animating the theory of de Rham Witt forms (\cref{derhamw}), we obtain the following:
\begin{proposition}
      Let $A$ be an $\F_p$-algebra. There is a descending exhaustive complete $\Z$-indexed filtration $\mathrm{Fil}^*_\mathrm{} \TR(A,p)$ on $\TR(A,p)$ such that $\mathrm{gr}^n_\mathrm{} \TR(A,p) \simeq \L W\Omega^n_A [n].$ 
\end{proposition}{}

 Therefore, we pursue the more general question of obtaining a filtered version of \cref{hesselholt}. More precisely, we prove the following:

\begin{theorem}\label{mainthm2intro}
  Let $S$ be a quasisyntomic $\F_p$-algebra. Then  we have a natural isomorphism $$ (\Fil^{*-1} \TR(S))[1] \simeq \varprojlim_k \mathrm{fib} \left(\Fil^* \TC(S[t]/t^k) \to \Fil^* \TC(S)\right).$$  
\end{theorem}{}
Note that passing to graded pieces yields \cref{mainthmintro}. In order to prove \cref{mainthm2intro}, we crucially use the technique of quasisyntomic descent introduced in \cite{BMS2}. This allows one to reduce to the case when $S$ is a quasiregular semiperfect algebra. In this case, the filtration $\Fil^* \TC(S)$ is understood concretely by the work of \cite{BMS2}. In \cref{cor3}, by studying the animated de Rham Witt forms, we prove that if $S$ is a quasiregular semiperfect algebra, then $\Fil^n \TR(S)$ is given by $\tau_{\ge 2n} \TR(S)$. Relatedly, we use the animated de Rham--Witt forms studied in our paper to give an entirely different, simpler, and more conceptual proof of the following result of Darrell and Riggenbach \cite[Thm.~A]{Noah}.

\begin{proposition}
    Let $S$ be a quasiregular semiperfect algebra. Then $\pi_* \TR(S)$ is concentrated in even degrees.
\end{proposition}{}
Back to \cref{mainthm2intro}, one further needs to understand $\varprojlim_{k} \Fil^* \TC (S[t]/t^k),$ when $S$ is quasiregular semiperfect. We show that this is given by the ``odd filtration": 
\begin{equation}\label{sharpbound}
     \varprojlim_k \Fil^n \TC (S[t]/t^k) \simeq \tau_{\ge 2n-1}  \varprojlim_k \TC (S[t]/t^k)
\end{equation}(see \cref{prop1}). In order to show this, we require certain estimates on $\varprojlim _k \Z_p(n)(S[t]/t^k).$ To this end, we show the following, which, along with \cref{filtrationlemma}, implies \cref{sharpbound}.

\begin{proposition}
Let $S$ be a quasiregular semiperfect ring. Then $$ \varprojlim_k \Z_p(i)(S[t]/t^k) \in D_{[-1,0]} (\Z_p).$$    
\end{proposition}{}

The arguments for proving the above proposition rely on studying properties of the Nygaard filtration on derived crystalline cohomology as well as the Hodge and conjugate filtration on derived de Rham cohomology. This result is the content of \cref{sec3}. In \cref{finalsec}, we put together the knowledge of all these filtrations to prove \cref{mainthm2intro}.

\subsection*{Conventions:} We will freely use the language of $\infty$-categories as in \cite{Lur09}, more specifically, the language of stable $\infty$-categories \cite{luriehigher}. For an ordinary commutative ring $R,$ we will let $D(R)$ denote the derived $\infty$-category of $R$-modules, so that it is naturally equipped with a $t$-structure and $D_{\ge 0}(R)$ (resp. $D_{\le 0} (R)$) denotes the connective (resp. coconnective) objects. For a map $W \to V$, $V/W$ will mean the cofiber unless otherwise mentioned. If $\mathcal C$ is a stable $\infty$-category, its associated category of pro-objects $\mathrm{Pro}(\mathcal{C})$ is also a stable $\infty$-category\footnote{Combine \cite[Rmk.~2.13]{LurX} and \cite[Thm.~4.5]{LurX} to obtain this statement.}. Let $\mathrm{Poly}_R$ denote the category of finitely generated polynomial algebras and $\mathrm{Alg}_R$ denote the category of ordinary commutative $R$-algebras. Then any functor $F: \mathrm{Poly}_R \to \mathcal{D}$ can be left Kan extended to a functor $\underline{F}: \mathrm{Ani}(\mathrm{Alg}_R) \to \mathcal{D},$ where $\mathrm{Ani}(\mathrm{Alg}_R)$ denotes the $\infty$-category of animated $R$-algebras. The functor $\underline{F},$ or $\underline{F} \mid_{\mathrm{Alg}_R}$ will be called the animation of $F.$ If $A$ is an algebra over a field of characteristic $p>0$, we will write $A^{(p^n)}$ to denote its $n$-fold Frobenius twist. $W(A)$ will denote the ring of $p$-typical Witt vectors.

\subsection*{Acknowledgement}We are very thankful to Bhargav Bhatt and Peter Scholze for helpful conversations and suggestions. We would further like to thank Luc Illusie, Akhil Mathew, Alexander Petrov, Noah Riggenbach, Arpon Raksit and Lucy Yang for helpful discussions regarding the paper. We are especially grateful to Luc Illusie for many detailed comments and questions on a draft version of this paper. The first author is supported by the PD Soros Fellowship and NSF DGE-2140743. The second author acknowledges support from the Max Planck Institute for Mathematics, Bonn and the University of British Columbia.

\section{Animated de Rham--Witt theory and filtration on $p$-typical $\TR$}\label{sec2}Let $k$ be a perfect field of characteristic $p>0.$ In this section, we discuss the ``motivic" filtration on $\mathrm{TR}(A,p)$ for a $k$-algebra $A.$ In order to do so, we discuss an animated form of the theory of de Rham--Witt forms from \cite{Luc1}. As demonstrated in \cite{Luc1}, the theory of Cartier operators (see \cite[\S~2]{Luc1}) play an important role in analyzing the de Rham--Witt forms via devissage. To this end, we will begin by discussing an animated form of Cartier operators, which would play a similar role in analyzing the animated de Rham--Witt forms. We will focus on understanding the animated Cartier operators in the case of quasiregular semiperfect algebras, which determines the entire theory via quasisyntomic descent.


\begin{construction}[Animated Cartier operators]
Let $A \in \mathrm{Alg}_k.$ For each $i \ge 0,$ we define $Z_n \L \Omega^i_A$ to be the animation of the functor $Z_n \Omega^i_{(\cdot)}: \mathrm{Poly}_k \to D(k)$ from \cite[\S~2.2.2]{Luc1}.
Similarly, we define $B_n \L \Omega^i_A$ to be the animation of the functor $B_n \Omega^i_{(\cdot)}: \mathrm{Poly}_k \to D(k)$ from \cite[\S~2.2.2]{Luc1}. There are canonical maps $B_n \L \Omega^i_A \to Z_n \L \Omega^i_A$ and $Z_n \L \Omega^i_A  \to \L \Omega^i_A.$ Animating the Cartier isomorphism (see \cite[\S~2]{Luc1}), we obtain a fiber sequence
\begin{equation}\label{buyfan}
    B_n \L \Omega^i_A \to Z_n \L \Omega^i_A \to \L\Omega^i _{A^{(p^n)}}.
\end{equation}By construction, for each $n \ge 0,$ we have the following natural fiber sequences in $D(\Z)$:
\begin{equation}\label{cart1}
  B_1 \LL \Omega^i_A \to  Z_{n+1} \L \Omega^i_A \xrightarrow[]{C} Z_{n} \L \Omega^i_{A}
\end{equation}\begin{equation}\label{cart2}
    B_1 \LL \Omega^i_A \to  B_{n+1} \L \Omega^i_A \xrightarrow[]{C} B_{n} \L \Omega^i_A.
\end{equation}{}
Note that we set $B_0 \L\Omega^i_A = 0$ and $Z_0 \L\Omega^i_A = \L\Omega^i_A.$ By construction, it follows that for all $n \ge 0$, we have $B_n \LL \Omega^0_A \simeq 0$ and $Z_n \LL \Omega^0_A \simeq A^{(p^n)}.$ Again, by construction, we have a fiber sequence
\begin{equation}\label{cartier3}
    Z_1 \L \Omega^i_A \to \L \Omega^i_A \xrightarrow[]{d} B_1 \LL \Omega^{i+1}_A.
\end{equation}

\end{construction}{}

\begin{proposition}
The functors from $\mathrm{Alg}_k^{\mathrm{op}}$ to $D(\Z)$ determined by $A \mapsto B_n \LL \Omega^i_A $ and $A \mapsto Z_n \LL \Omega^i_A$ are fpqc sheaves for all $n \ge 0$. 
\end{proposition}{}

\begin{proof}
  Note that the claim holds when $i=0.$ Let us suppose that for a fixed $i \ge 0,$ we have shown that the functors $A \mapsto B_n \LL \Omega^i_A $ and $A \mapsto Z_n \LL \Omega^i_A$ are fpqc sheaves. Since the functor $A \mapsto \LL\Omega_A^j$ satisfies fpqc descent for all $j$ (see \cite[Thm.~3.1]{BMS2}), by \cref{cartier3}, $A \mapsto B_1 \LL \Omega_A^{i+1}$ is an fpqc sheaf. By \cref{cart2}, $A \mapsto B_n \LL \Omega_A^{i+1}$ is an fpqc sheaf for all $n.$ By \cref{buyfan}, $A \mapsto Z_n \LL \Omega_A^{i+1}$ is an fpqc sheaf for all $n.$ Therefore, by induction on $i$, we are done.
\end{proof}{}

\begin{remark}
For $A \in \mathrm{Alg}_k,$ the objects $Z_n \LL \Omega^i_A$ and $B_n \LL \Omega^i_A$ can be naturally viewed as objects of $D(A^{(p^n)}).$ The fiber sequence \cref{buyfan} lifts to a fiber sequence in $D(A^{(p^n)}).$
\end{remark}{}

In order to further understand $Z_1\LL \Omega^i_A$ via descent, it would be useful to relate it to the conjugate filtration $\mathrm{Fil}^*_{\mathrm{conj}} \LL \Omega^*_A$ and the Hodge filtration $\Fil^*_{\Hodge} \LL \Omega^*_A$ on derived de Rham cohomology. Recall that when $A$ is a polynomial algebra, then $Z_1\LL \Omega^i_A \simeq Z_1  \Omega^i_A := \mathrm{Ker}(\Omega^i_A \xrightarrow{d} \Omega^{i+1}_A).$ We begin by observing the following:

\begin{proposition}\label{intersection}
Let $A$ be a polynomial algebra over $k.$ Then $$Z_1\Omega^i_A[-i] \simeq \Fil^i_{\conj} \LL \Omega^*_A \times_{\LL \Omega^*_A} \Fil^i_{\Hodge} \LL \Omega^*_A.$$
\end{proposition}{}

\begin{proof}
    Since $A$ is a polynomial algebra over $k,$ we have $\Fil ^i_{\conj} \LL \Omega^*_A \simeq \tau_{\ge -i} \Omega^*_A$ and $\Fil^i_{\Hodge} \LL \Omega_A^* \simeq \Omega_A^{\ge i}.$ Note that there is a natural map $Z_1\Omega_A^i[-i] \simeq \tau_{\ge -i} \Omega_A^{\ge i} \to \Fil^i_{\conj} \LL \Omega^*_A \times_{\LL \Omega^*_A} \Fil^i_{\Hodge} \LL \Omega^*_A. $ Since we have a fiber sequence $\Omega_A^{\ge i} \to \Omega_A^* \to \Omega_A^{\le i-1},$ we obtain a natural isomorphism $$\Fil^i_{\conj} \LL \Omega^*_A \times_{\LL \Omega^*_A} \Fil^i_{\Hodge} \LL \Omega^*_A \xrightarrow{\simeq} \mathrm{fib} (\Fil^i_{\conj} \LL \Omega_A^* \to \Omega_A^{\le i-1}).$$ By computing homotopy groups, one sees that the map $Z_1\Omega_A^i[-i] \to \mathrm{fib} (\Fil^i_{\conj} \LL \Omega_A^* \to \Omega_A^{\le i-1})$ is an isomorphism, which finishes the proof.
\end{proof}{}
By animating \cref{intersection}, we obtain the following:
\begin{corollary}\label{sunday}
    Let $A \in \mathrm{Alg}_k.$ Then $$Z_1\LL\Omega^i_A[-i] \simeq \Fil^i_{\conj} \LL \Omega^*_A \times_{\LL \Omega^*_A} \Fil^i_{\Hodge} \LL \Omega^*_A.$$
\end{corollary}{}

\begin{proposition}\label{perfectringcart}
    Let $A$ be a perfect ring. Then $B_n \LL \Omega^i_A \simeq Z_n \LL \Omega^i_A \simeq 0$ for all $i >0.$
\end{proposition}{}

\begin{proof}
    Since $A$ is perfect, we have $\LL \Omega_A^i =0$ for $i>0.$ Therefore, using \cref{buyfan}, it suffices to prove that $B_n \LL \Omega^i_A \simeq 0$ for $i>0.$ By \cref{cart2}, this reduces to the case $n=1.$ Using \cref{cartier3} and descending induction on $i,$ it suffices to show that the map $\mathrm{can}: Z_1 \LL \Omega^0_A \to \LL \Omega^0_A$ is an isomorphism, which follows because $A$ is perfect.
\end{proof}{}

Before proceeding further, let us recall two definitions from \cite{BMS2}.

\begin{definition}[{\cite[Def.~4.10]{BMS2}}]
    Let $f: A \to B$ be a map of $\F_p$-algebras. Then $f$ is said to be quasisyntomic if it is flat and the cotangent complex $\LL_{B/A}$ has Tor amplitude in homological degrees $[0,1].$
\end{definition}{}

\begin{proposition}[{\cite[Def.~8.8]{BMS2}}]
    An $\F_p$-algebra $S$ is said to be semiperfect if the natural map $S^\flat:= \varprojlim_{x \to x^p} S \to S$ is surjective. $S$ is called quasiregular semiperfect if $S$ is semiperfect and $\L_{S/\F_p}$ is a flat $S$-module supported in homological degree 1.
\end{proposition}{}

\begin{proposition}\label{coollemma}
    Let $A$ let a quasiregular semiperfect algebra. Then $Z_1 \LL \Omega^i_A[-i]$ is discrete for all $i \ge 0.$
\end{proposition}{}

\begin{proof}
    Let $I:= \mathrm{Ker}(A^\flat \to A).$ By \cite[Prop.~8.12]{BMS2}, one knows that $\LL \Omega^*_A \simeq D_{A^\flat}(I).$ Further, the conjugate filtration $\mathrm{Fil}^i_{\mathrm{conj}}\LL \Omega^*_A$ is also discrete and is given by the $A^\flat$-submodule of $D_{A^\flat}(I)$ generated by elements of the form $a_1^{[l_1]} \cdots  a_m^{[l_m]}$ such that $m \ge 0,$ $a_i \in I$ and $\sum _{u=1}^{m} l_u < (i+1)p.$ The Hodge filtration on $\LL\Omega^*_A$ identifies with the divided power filtration on $D_{A^\flat}(I),$ i.e., $\Fil^i_{\Hodge} \LL\Omega_A^*$ is the ideal generated by elements of the form $a_1^{[l_1]} \cdots  a_m^{[l_m]}$ such that $m \ge 0,$ $a_i \in I$ and $\sum _{u=1}^{m} l_u \ge i.$ In particular, note that the map $\Fil^i_{\Hodge} \LL\Omega^*_A \to \LL \Omega_A^*$ is injective. 

From the above description, we see that the composite map $$\Fil^i_{\conj} \LL\Omega^*_A \to \LL\Omega^*_A \to \LL\Omega^*_A / \Fil^i_{\Hodge} \LL\Omega^*_A$$ is surjective. This implies that the fiber of $\Fil^i_{\conj} \LL\Omega^*_A  \to \LL\Omega^*_A / \Fil^i_{\Hodge} \LL\Omega^*_A$ must be discrete. By \cref{sunday}, the fiber identifies with $Z_1\LL\Omega_A^i [-i],$ which finishes the proof.
\end{proof}

\begin{proposition}\label{prop45}
    Let $A$ be a quasiregular semiperfect algebra. Then $B_n \LL \Omega^i_A[-i]$ and $Z_n \LL \Omega_A^i[-i]$ are discrete for all $n, i \ge 0.$
\end{proposition}{}

\begin{proof}
Note that $B_0 \LL \Omega^i_A = 0$ by construction, and $Z_0 \LL \Omega^i_A [-i]\simeq \LL \Omega^i_A[-i] \simeq \Gamma^i_A (I/I^2),$ where $I := \mathrm{Ker}(A^\flat \to A),$ thus the claim holds when $n=0.$ By the proof of \cref{coollemma}, we see that if $A$ is a quasiregular semiperfect algebra, one has a natural identification $Z_1 \LL\Omega_A^i[-i] \simeq \Fil^i_{\conj} \LL \Omega^*_A \cap \Fil^i_{\Hodge} \LL \Omega^*_A$ owing to the discreteness of all objects involved. The explicit description of $\Fil^i_{\conj} \LL \Omega^*_A $ and $\Fil^i_{\Hodge} \LL \Omega^*_A$ in this case implies that the natural map $Z_1 \LL \Omega^i_A[-i] \to \mathrm{gr}^i_{\conj} \LL \Omega^*_A$ is surjective. Using the fiber sequence \cref{buyfan}, it follows that $B_1 \LL \Omega^i_A[-i]$ is discrete. By \cref{cart2}, it inductively follows that $B_n \LL \Omega^i_A[-i]$ is discrete for all $n \ge 1.$ Applying \cref{buyfan} again, and using the fact that $\LL \Omega^i_A [-i] $ is discrete, we see that $Z_n \LL \Omega^i_A[-i]$ is discrete for all $n \ge 1.$
\end{proof}{}

Now, we will discuss animation of the de Rham--Witt theory from \cite{Luc1}.

\begin{definition}
    Let $\mathrm{Alg}_k$ denote the category of $k$-algebras. For each $i \ge 0,$ we define $$\L W_n \Omega^i_{(\cdot)}\colon \mathrm{Alg}_k \to D(W_n(k))$$ to be the animation of the functor $W_n \Omega^i_{(\cdot)}\colon \mathrm{Poly}_k \to D(W_n(k))$ defined in \cite{Luc1}.
\end{definition}{}
For $A \in \mathrm{Alg}_k,$ the object $\LL W_n \Omega^i_A$ is naturally an object of $D(W_n(A))$; further, the usual operators on the de Rham--Witt complex extend to $\L W_n \Omega^i_A.$ In particular, we have $F\colon \L W_n \Omega^i _A\to \L W_{n-1} \Omega^{i}_A$, $V\colon \L W_n \Omega^i_A\to \L W_{n+1} \Omega^i_A$ which extend the Frobenius and Verschiebung maps. We also have a map $R\colon \L W_n \Omega^i_A \to \L W_{n-1} \Omega^{i}_A $ extending the restriction maps which allows one to obtain a pro-object $\L W_{\smallbullet} \Omega^i_{A}$ in the derived $\infty$-category of $W(k)$-modules. The operators $F$ and $V$ may be lifted to maps $F \colon \L W_{\smallbullet} \Omega^i_{A} \to \L W_{{\smallbullet}-1} \Omega^i _{A}$ and $V \colon \L W_{\smallbullet} \Omega^i_{A} \to \L W_{{\smallbullet}+1} \Omega^i _{A}.$ When $A$ is a polynomial algebra, it follows that $FV= VF= p;$ by animation, this gives natural fiber sequences of pro-objects
\begin{equation}\label{eq100}
   \L W_{\smallbullet} \Omega^i_A/F \to \L W_{\smallbullet} \Omega^i_A/p \to \L W_{\smallbullet} \Omega^i_A/V
\end{equation}{}
and
\begin{equation}\label{eq101}
   \L W_{\smallbullet} \Omega^i_A/V \to\L W_{\smallbullet} \Omega^i_A/p \to \L W_{\smallbullet} \Omega^i_A/F.
\end{equation}{}
\begin{definition}\label{derhamw}
    Let $\mathrm{Alg}_k$ denote the category of $k$-algebras. For each $i \ge 0,$ we define $$\L W \Omega^i_{(\cdot)}\colon \mathrm{Alg}_k \to D(W(k))$$ to be the functor determined by sending a $k$-algebra $A$ to $\L W\Omega^i_A\colon= \varprojlim _{n,R} \L W_n \Omega^i_A.$ 
\end{definition}{}
\begin{example}
    Note that $\LL W_n \Omega_A^0 \simeq W_n(A)$ and therefore, $\LL W\Omega_A^0 \simeq W(A),$ the usual ring of $p$-typical Witt vectors.
\end{example}{}

We have natural maps $F,V\colon \L W\Omega^i_A \to \L W\Omega^i_A$ which are obtained by passing to the limit of the map of pro-objects above. We also obtain the following fiber sequences

\begin{equation}\label{eq200}
    \L W \Omega^i_A /F \to \L W \Omega^i_A/p \to \L W \Omega^i_A/V
\end{equation}and
\begin{equation}\label{eq201}
 \L W \Omega^i_A/V \to \L W \Omega^i_A/p \to \L W \Omega^i_A /F.
\end{equation}{}

The following proposition will allow us to calculate the animated de Rham--Witt sheaves by devissage.

\begin{proposition}\label{Illusieprop}
Let $A \in \mathrm{Alg}_k.$ For $r \ge 1,$ we have a fiber sequence
$$ \L W\Omega^{i-1}_A/F^r \to \L W\Omega^i_A/V^r \to \L W_r\Omega^i_A  $$
\end{proposition}{}

\begin{proof}
    When $r=1,$ the proposition follows from animating \cite[Prop.~3.18]{Luc1} and passing to inverse limits. In general, our claim is a consequence of the following:
\end{proof}{}

\begin{lemma}
    Let $A \in \mathrm{Alg}_k$ Then we have a fiber sequence of pro-objects
$$ \LL W_{{\smallbullet}-r}\Omega^{i-1}_A/F^r  \xrightarrow[]{dV^r} \LL W_{{\smallbullet}}\Omega^i_A/V^r \to \LL W_r\Omega^i_A  $$
\end{lemma}

\begin{proof}
Suppose that $A$ is a polynomial algebra. Let $\underline{R}: W_{n+1} \Omega^i_A [p^r] \to W_{n} \Omega^i_A [p^r]$ be the induced restriction maps. By \cite[Prop.~3.4]{Luc1}, $\underline{R}^r = 0;$ in particular, the same holds for $W_{n+1}\Omega^i_A[F^r]$ and $W_{n+1}\Omega^i_A[V^r]$. Therefore, the desired result follows from animation and passing to pro-objects, once we show that for a polynomial algebra $A,$ the following sequence is exact for $n \ge 2r,$ where the quotients are taken in the discrete sense:
 $$0 \to V^r W_{n-r} \Omega^{i-1}_A /p^r W_n \Omega^{i-1}_A \xrightarrow[]{d} W_n \Omega^i_A/V^r W_{n-r}\Omega^i_A \to W_r\Omega^i_A \to 0.$$
 Using \cite[{Prop.~3.2}]{Luc1}, one obtains the exactness in the middle and right. For the injectivity of $d,$ let us suppose that there is an $x \in W_{n-r}\Omega^{i-1}_A$ such that $dV^r x = V^r y$ for some $y \in W_{n-r}\Omega^i_A.$ Applying $F^r$ on both sides (also using $FdV= d$ and $FV= p$), we obtain $dx = p^r y.$ Let $\overline{x} \in W_r \Omega^{i-1}_A$ be the restriction of $x.$ Therefore, we have $d \overline{x} = 0.$ It follows from \cite[Prop.~3.21]{Luc1}\footnote{Illusie pointed out that there was a gap in the proof of \cite[Prop.~3.21]{Luc1} which was fixed in \cite[II,~1.3]{Luc11}.} that there exists an $\alpha \in W_{2r}\Omega^{i-1}_A$ such that $\overline{x} = F^r \alpha.$
 Now let $\widetilde{\alpha}, \widetilde{x} \in W\Omega^{i-1}_A$ be elements that restrict to $\alpha, x$ respectively.
Then it follows that $\widetilde{x} - F^r \widetilde{\alpha}$ restricts to zero in $W_r \Omega^{i-1}_A.$ Since the restriction maps induce a quasi-isomorphism of complexes $W\Omega^*_A/p^r \to W_r\Omega^*_A$ (see \cite[Cor.~3.17]{Luc1}), it follows that $\widetilde{x} - F^r\widetilde{\alpha} \in p^r W\Omega^{i-1}_A + d W\Omega^{i-2}.$ This implies that $V^r \widetilde{x} = p^r z$ for some $z \in W\Omega^{i-1}_A.$ By restriction, we obtain $V^rx = p^r \overline{z}$, where $\overline{z} \in W_n\Omega^{i-1}_A.$ This finishes the proof.
\end{proof}{}

To further analyze the animated de Rham--Witt sheaves, we use the following notion from \cite{Ben}.

\begin{definition} Let $M \in D(\Z)$ and $V: M \to M$ be an endomorphism. We will say that $M$ is derived $V$-complete if $\varprojlim_V M =0.$
\end{definition}{}

Let $M/V^n:= \mathrm{cofib} (M \xrightarrow[]{V^n}M).$ Then we have a fiber sequence $\varprojlim_{V}M \to M \to  \varprojlim M/V^n,$ which implies that $M$ is derived $V$-complete if and only if the natural map $M \to \varprojlim M/V^n$ is an isomorphism. The $\infty$-category of derived $V$-complete objects is closed under limits.

\begin{lemma}\label{lemma1}
    Let $M \in D(\Z)$ and $V: M \to M$ be such that $M$ is derived $V$-complete. If $M/V$ is discrete then so is $M.$
\end{lemma}{}

\begin{proof}
We will prove that $M/V^n$ is discrete by induction on $n$.
The diagram $M \xrightarrow{V} M \xrightarrow{V^{n-1}}M $ yields a fiber sequence $$ M/V \to M/V^n \to M/V^{n-1}.$$ It follows inductively from the hypothesis that $M/V^n$ is discrete for all $n \ge 1$ and the natural maps $M/V^n \to M/V^{n-1}$ are surjective. Since $M \simeq \varprojlim M/V^n$ by derived $V$-completeness of $M$, the conclusion follows.
\end{proof}{}

\begin{lemma}\label{derV}
    Let $A \in \mathrm{Alg}_k.$ Then $\L W \Omega^i_A$ is derived $V$-complete for each $i \ge 0.$
\end{lemma}{}

\begin{proof}
    Let us define $V'\colon \L W_{n} \Omega^i_A \to \L W_{n} \Omega^i_A$ as the composite of $V\colon \L W_n\Omega^i_A \to \L W_{n+1} \Omega^i_A$ with $R\colon \L W_{n+1} \Omega^i_A \to \L W_{n} \Omega^i_A.$ When $A$ is a polynomial algebra, it follows that $V'^n=0.$ Therefore, $\L W_n \Omega^i_A$ is derived $V'$-complete after animation. Passing to inverse limit over $n$, it follows that $\L W\Omega^i_A$ is derived $V$-complete.
\end{proof}{}

\begin{proposition}\label{flatdescent} The functor $\mathrm{Alg}_k^{\mathrm{op}} \to D(\Z)$ determined by $A \mapsto \L W \Omega^i_A$ is an fpqc sheaf.
\end{proposition}{}

\begin{proof}
We will use induction on $i;$ the case $i=0$ is clear. By \cref{derV}, we reduce to checking that $A \mapsto \L W\Omega^i_A/V$ is an fpqc sheaf. The latter follows inductively from \cite[Theorem 3.1]{BMS2} and \cref{Illusieprop} (in the case $r=1$).
\end{proof}{}

\begin{remark}
    The functor $\mathrm{Alg}_k^{\mathrm{op}} \to D(\Z)$ determined by $A \mapsto \L W_r \Omega^i_A$ is also an fpqc sheaf; this follows from \cref{Illusieprop} and \cref{flatdescent}.
\end{remark}{}

The importance of Cartier operators for our purposes is reflected in the following proposition.

\begin{proposition}\label{llll}
    Let $A \in \mathrm{Alg}_k.$ Then $ \LL W \Omega^i_A/ V \simeq \varprojlim_{C} Z_n \LL \Omega^i_A$ and $\LL W\Omega^i_A/ F \simeq \varprojlim_{C} B_{n} \LL \Omega^{i+1}_A.$
\end{proposition}{}

\begin{proof}
Let $A$ be a polynomial algebra over $k.$ Then by \cite[Prop. 3.11]{Luc1}, there is a natural map $\mathrm{cofib}(W_{n} \Omega^i_A \xrightarrow{V} W_{n+1} \Omega^i_A  ) \xrightarrow{F^n} Z_n  \Omega^i_A$ of $\mathbb{N}$-indexed (via the restriction maps $W_{n+1}\Omega^i_A \to W_{n}\Omega^i_A$ on the left-hand side and via $C$ on the right-hand side) objects whose fiber is an $\mathbb{N}$-indexed object whose transition maps are all zero. By animation and taking inverse limit over $n \in \mathbb{N},$ we obtain $$ \LL W \Omega^i_A/ V \simeq \varprojlim_{C} Z_n \LL \Omega^i_A.$$ The other assertion is deduced similarly from loc. cit.
\end{proof}{}

\begin{proposition}\label{perfectwitt}
Let $S$ be a perfect ring. Then $\LL W \Omega^i_S = 0$ for $i>0.$    
\end{proposition}{}
\begin{proof}
Follows from \cref{perfectringcart} and \cref{llll}.     
\end{proof}{}

    

\begin{proposition}\label{fixthis}
    Let $S$ be a quasiregular semiperfect algebra. Then $\L W\Omega^i_S[-i]$ is discrete for each $i \ge 0$.
\end{proposition}{}

\begin{proof}
Using \cref{lemma1}, it suffices to prove that $(\L W \Omega^i_S/V) [-i]$ is discrete. By \cref{llll}, $\L W \Omega^i_S/V \simeq \varprojlim  Z_n \LL \Omega^i_A.$ By \cref{prop45}, $ Z_n \LL \Omega^i_A[-i]$ is discrete. Since (by \cref{prop45}), $B_1 \LL \Omega^i_A[-i]$ is also discrete, using the fiber sequence \cref{cart1}, we see that the maps $ Z_{n+1} \LL \Omega^i_A[-i] \xrightarrow{C}  Z_n \LL \Omega^i_A[-i]$ are surjections. This finishes the proof. 
\end{proof}{}

Finally, we begin our discussion on $\TR.$

\begin{construction}\label{cons1}
    Note that for a $k$-algebra $A,$ the functor $A \mapsto \TR ^r (A,p)$ is the animation of its restriction to the full subcategory of finitely generated polynomial algebras. Further, when $A$ is smooth, by \cite[Thm.~B]{Hess}, we have
$$\pi_n \TR^r(A,p) \simeq \bigoplus_{\substack{0 \le i \le n\\ n-i\,\w{is even}}} W_r\Omega^i_A.$$ Therefore, by animating the decreasing Postnikov filtration on $\TR^r(A,p)$ (given by $\tau_{\ge *}\TR^r(A,p)$) from the category of polynomial algebras, one can equip $\TR^r(A,p)$ with the structure of a filtered object $\mathrm{Fil}^*_{\mathrm{}} \TR^r(A,p)$ such that 
\begin{equation}\label{eq1}
\mathrm{gr}^n_{\mathrm{M}} \TR^r(A,p) \cong
    \displaystyle{\bigoplus_{\substack{0 \le i \le n\\ n-i\,\w{is even}}}\L W_r\Omega^i_A[n]}
\end{equation}
Note that there are natural restriction maps $R\colon \TR^{r+1}(A,p) \to \TR^{r}(A,p)$, and one sets $$\TR (A,p) \colon= \varprojlim_{r,R} \TR^r(A,p).$$ By passing to the inverse limit over the restriction maps one can equip $\TR(A,p)$ with the structure of a filtered object $\mathrm{Fil}^*_{\mathrm{M}} \TR(A,p)$. Note that by \cite[Thm.~B]{Hess}, the map $R\colon \TR^{r+1}(A,p) \to \TR^{r}(A,p)$ induces a map $R_{n,r,i}\colon \L W_{r+1} \Omega^i_A[n] \to \L W_{r} \Omega^i_A [n]$ at the level of the $i$-th summand of $\mathrm{gr}^n$ that is equivalent to $(p \lambda_{r+1})^{\frac{n-i}{2}}R[n]$, where $\lambda_{r+1} \in (\Z/p^{r+1} \Z)^\times.$ It follows from this description that $\varprojlim _{r,R_{n,r,i}} \L W_{r}\Omega^i_A [n] \simeq 0$ if $n>i.$ Therefore, we see that $\mathrm{gr}^n \TR(A,p) \simeq \L W\Omega^n_A [n].$ Let us summarize this construction in the following proposition.
\end{construction}{}

\begin{proposition}\label{trr}
  Let $A$ be a $k$-algebra. There is a descending exhaustive complete $\Z$-indexed filtration $\mathrm{Fil}^* \TR(A,p)$ on $\TR(A,p)$ such that $\mathrm{gr}^n \TR(A,p) \simeq \L W\Omega^n_A [n].$ This may be called the ``motivic" filtration on $\TR(A,p).$
\end{proposition}{}

\begin{proof}
    The construction of the filtration and the description of the graded pieces follow from the above discussion. In order to prove the completeness of the filtration $\Fil^* \TR(A,p)$, by construction, it would be enough to show the completeness of $\Fil^*\TR^r (A,p).$ To this end, it is enough to show that $\Fil^k \TR^r(A,p)$ is $k$-connective. By considering sifted colimits, this reduces to the case when $A$ is a polynomial algebra, in which case the result follows since the filtration is given by the Postnikov filtration.
\end{proof}{}

\begin{corollary}[{\cite[Thm.~6.14]{Ben}}]Let $S$ be a perfect ring. Then $\TR(S,p) \simeq W(S).$
\end{corollary}{}

\begin{proof}
Follows from \cref{perfectwitt} and \cref{trr}. 
\end{proof}{}

Let us now consider the category of quasisyntomic $k$-algebras $\mathrm{qSyn}_k,$ thought of as a Grothendieck site equipped with the quasisyntomic topology. If $A$ is in $\mathrm{qSyn}_k,$ we will give a different construction of $\Fil^* \TR(A,p)$ by quasisyntomic descent that will be important later in this paper. By \cite[Prop.~4.31]{BMS2} any quasisyntomic sheaf on $\mathrm{qSyn}_k$ is determined by its values on quasiregular semiperfect algebras.

\begin{proposition}\label{descent}
    Let $A \in \mathrm{qSyn}_k.$ The functor $A \mapsto \mathrm{Fil}^* \TR(A,p)$ is a quasisyntomic sheaf with values in filtered spectra. 
\end{proposition}{}

\begin{proof}
Since $\Fil^* \TR (A,p)$ is a complete descending filtration on $\TR(A,p)$, by considering limits, it would be enough to prove that $A \mapsto \mathrm{gr}^n\TR(A,p) \simeq \L W\Omega^n_A [n] $ is a sheaf for all $n.$ The latter follows from \cref{flatdescent}.
\end{proof}{}

\begin{proposition}\label{evenn}
    Let $R$ be a quasiregular semiperfect algebra. Then $\pi_* \TR(R,p)$ is concentrated in even degrees.
\end{proposition}{}
\begin{proof}
 Follows from \cref{trr} and \cref{fixthis}.
\end{proof}{}
The following proposition, along with \cref{descent}, gives an alternative way to understand the filtration on $p$-typical $\TR$ via quasisyntomic descent (see \cite[Prop.~4.31]{BMS2}).
\begin{proposition}\label{doublespeed}
Let $R$ be a quasiregular semiperfect algebra. Then $$\Fil^n \TR(R,p) \simeq \tau_{\ge 2n} \TR(R,p).$$
\end{proposition}{}
\begin{proof}
    Follows because $\gr^n \TR(R,p)[-2n]$ is discrete.
\end{proof}

Note that for any $\F_p$-algebra $A,$ there is a canonical product decomposition $\TR(A) \simeq \prod_{(k,p)=1} \TR(A,p).$ One may define a filtration 
$ \Fil^* \TR(A) \coloneqq \prod_{(k,p)=1} \Fil^* \TR(A,p).$ This equips $\TR(A)$ with a descending complete exhaustive filtration such that $\gr^n \TR(A) \simeq \prod_{(k,p)=1}\L W\Omega^n_A [n].$ Our previous discussion on $\TR(A,p)$ implies the following corollaries.

\begin{corollary}\label{cor1}
    Let $A \in \mathrm{qSyn}_k.$ The functor $A \mapsto \mathrm{Fil}^* \TR(A)$ is a quasisyntomic sheaf with values in filtered spectra. 
\end{corollary}{}

\begin{corollary}\label{cor3}
    Let $R$ be a quasiregular semiperfect algebra. Then $\pi_* \TR(R)$ is concentrated in even degrees, and $$\Fil^n \TR(R) \simeq \tau_{\ge 2n} \TR(R).$$
\end{corollary}{}

\section{Pro-system of truncated polynomial rings}\label{sec3} 
Let $S$ be a quasisyntomic $\F_p$-algebra. In \cite{BMS2}, Bhatt--Morrow--Scholze constructed a ``motivic" filtration $\Fil^* \TC (S)$ on $\TC(S)$ where the graded pieces $\mathrm{gr}^n \TC(S)$ are given by $\Z_p (n)(S)[2n].$ The goal of this section is to identify the induced motivic filtration on $\varprojlim \mathrm{TC}(R[t]/t^k)$, when $R$ is a quasiregular semiperfect algebra, with the ``odd filtration" (see \cref{prop1}). We will use the description of the Tate twists $\Z_p(n)$ in terms of Nygaard filtration on derived crystalline cohomology. To this end, we recall a few notations and basic properties.

\begin{notation}Let $A$ be an $\F_p$-algebra.
   We will use $\dR$, $\Fil^*_{\Hodge} \dR$ and $\Fil^*_{\conj} \dR$ to denote the functors $\L \Omega^*_{(\cdot)},$ $\Fil^*_{\Hodge}\L\Omega^*_{(\cdot)}$ and $\Fil^*_{\conj}\L\Omega^*_{(\cdot)}$. Let $\widehat{\dR}$ denote the completion of $\dR$ with respect to $\mathrm{Fil}^*_{\Hodge} \dR$, so that it is naturally equipped with a filtration $\Fil^*_{\Hodge} \widehat {\dR}.$ By construction, $\mathrm{gr}^n_{\Hodge} \dR(A) \simeq \mathrm{gr}^n_{\Hodge} \widehat{\dR}(A) \simeq \wedge^n\LL_{A/\F_p}[-n].$ By animating the Cartier isomorphism, $\mathrm{gr}^n_{\conj} \dR(A) \simeq \wedge^n \LL_{A^{(p)}/ \F_p}[-n]$ (see \cite[Prop.~3.5]{Bha12}). The following proposition discusses certain monoidal properties of these functors that will be useful later. 
\end{notation}{}

\begin{proposition}\label{convv}
 Let $A$ and $B$ be two $\F_p$-algebras. Then,
\begin{enumerate}
    \item $\dR(A \otimes_{\F_p} B) \simeq \dR(A) \otimes_{\F_p} \dR(B).$

\item $\Fil^n_{\Hodge}\dR(A \otimes_{\F_p} B) \simeq \colim_{j+k \ge n} \Fil^{j}_{\Hodge} \dR(A) \otimes_{\F _p} \Fil^{k}_{\Hodge} \dR(B).$

\item $\Fil^n_{\conj}\dR(A \otimes_{\F_p} B) \simeq \colim_{j+k \le n} \Fil^{j}_{\conj} \dR(A) \otimes_{\F _p} \Fil^{k}_{\conj} \dR(B).$
\end{enumerate}{}
\end{proposition}{}
\begin{proof}
    By animation, these can be checked by reducing to polynomial algebras. For polynomial algebras, one can further reduce it to checking on graded pieces and use \cite[Lem.~5.2]{BMS2} (\textit{cf.}~\cite{fild}).
\end{proof}{}

\begin{remark}
In the language of filtered derived categories $DF(\F_p)$ as in \cite{BMS2}, the construction appearing in the right hand side is called the Day convolution, which turns $DF(\F_p)$ into a symmetric monoidal stable $\infty$-category. One also has the completed filtered derived category $\widehat{DF}(\F_p)$, equipped with an induced monoidal structure. It follows from \cref{convv} that $\Fil^*_{\Hodge}\widehat{\dR}(A \otimes_{\F_p} B) \simeq \Fil^*_{\Hodge}\widehat{\dR}(A) \hat{\otimes} \Fil^*_{\Hodge}\widehat{\dR}(B),$ where the right hand side denotes the monoidal operation on $\widehat{DF}(\F_p).$
\end{remark}{}
\begin{notation}
Let $A$ be an $\F_p$-algebra. We let $R\Gamma_{\crys}(A)$ denote derived crystalline cohomology, and $\Fil^*_{\mathrm{Nyg}} R\Gamma_{\crys}(A)$ denote the Nygaard filtration; they are both defined to be animated from polynomial algebras. The associated Nygaard completed object will be dentoted by $\widehat{R\Gamma}_{\crys} (A)$, which is naturally equipped with a filtration $\Fil^*_{\mathrm{Nyg}} \widehat{R\Gamma}_{\crys} (A).$ We will only apply these notions in the case when $A$ is a quasisyntomic $\F_p$-algebra, and we will assume $A$ to be quasisyntomic from now for simplicity. The proposition below lists some basic properties of the Nygaard filtration.
\end{notation}{}

\begin{proposition}
    Let $A$ be a quasisyntomic $\F_p$-algebra. Then,

\begin{enumerate}

\item $\widehat{R\Gamma}_{\crys}(A)/p \simeq \widehat{\dR}(A).$

\item $\mathrm{gr}^n_{\mathrm{Nyg}} R\Gamma_{\crys}(A) \simeq \mathrm{gr}^n_{\mathrm{Nyg}} \widehat{R\Gamma}_{\crys}(A) \simeq \Fil^n_{\conj} \dR(A). $

\item Multiplication by $p$ induces a natural map $p: 
\Fil^{n-1}_{\mathrm{Nyg}}R\Gamma_{\crys}(A) \to \Fil^{n}_{\mathrm{Nyg}}R\Gamma_{\crys}(A)$ whose cofiber is naturally isomorphic to $\Fil^n_{\Hodge} \dR(A).$

\item Multiplication by $p$ induces a natural map $p: 
\Fil^{n-1}_{\mathrm{Nyg}}\widehat{R\Gamma}_{\crys}(A) \to \Fil^{n}_{\mathrm{Nyg}}\widehat{R\Gamma}_{\crys}(A)$ whose cofiber is naturally isomorphic to $\Fil^n_{\Hodge} \widehat{\dR}(A).$

\item There is a divided Frobenius map $\varphi_n: \mathrm{Fil}^n_{\mathrm{Nyg}} R\Gamma_{\crys} (A) \to R\Gamma_{\crys}(A)$ which gives a fiber sequence
$$ \Z _p(n)(A) \xrightarrow{}\mathrm{Fil}^n_{\mathrm{Nyg}}{R\Gamma}_{\crys} (A) \xrightarrow{\varphi_n - \mathrm{can}}  R\Gamma_{\crys}(A).$$

\item There is a divided Frobenius map $\varphi_n: \mathrm{Fil}^n_{\mathrm{Nyg}} \widehat{R\Gamma}_{\crys} (A) \to \widehat{R\Gamma}_{\crys}(A)$ which gives a fiber sequence
$$ \Z _p(n) (A)\xrightarrow{}\mathrm{Fil}^n_{\mathrm{Nyg}}\widehat{R\Gamma}_{\crys}(A)  \xrightarrow{\varphi_n - \mathrm{can}}  \widehat{R\Gamma}_{\crys}(A).$$
\end{enumerate}{}
\end{proposition}{}

\begin{proof}
    See \cite[\S~8]{BMS2} for the case when $A$ is a quasiregular semiperfect $\F_p$-algebra. The proposition follows by quasisyntomic descent.
\end{proof}{}

Let us define $\F_p (n) (A) := \Z_p (n)(A)/p.$ We show that $\F_p (n) (A)$ may be described purely in terms of the animated Cartier theory from \cref{sec2}.

\begin{proposition} Let $A$ be quasisyntomic $\F_p$-algebra. Then we have a natural fiber sequence
    \begin{equation}\label{cartierbloch}
    \F_p(n)(A)[n] \to Z_1 \LL \Omega^n_A \xrightarrow[]{\mathrm{can} - C} \LL\Omega^n_A.
\end{equation}
\end{proposition}{}

\begin{proof}
By animation, it is enough to prove the claim when $A$ is a polynomial algebra. By \cite[Prop.~8.21]{BMS2} and quasisyntomic descent, it follows that $$\Z_p (n)(A)[n] \simeq R\Gamma_{\mathrm{pro\acute{e}t}}(\Spec A, W\Omega^n_{\Spec A, \mathrm{log}}),$$ where $W\Omega^n_{\Spec A, \mathrm{log}} := \varprojlim W_r\Omega^n_{\Spec A, \mathrm{log}}$ (see \cite[Prop.~8.4]{BMS2}). The claim now follows from going modulo $p$, and using \cite[2.1.20,~2.4.1.1,~Thm.~2.4.2,~Cor.~5.7.5]{Luc1}.
\end{proof}{}

\begin{remark}\label{BMSsimple}
   Let $R$ be a quasiregular semiperfect algebra. By \cite[Lem.~8.19]{BMS2}, $\Z_p(i) (R)$ is discrete for $i>0.$ This may also be seen by reducing modulo $p$ and using the sequence \cref{cartierbloch}. Further, $\Z_p (0)(R) \in D_{[-1,0]}(\Z_p)$ and $\Z_p (i) (R) = 0$ for $i<0.$ Using \cref{filtrationlemma}, we see that the filtration $\Fil^n\TC(R)$ constructed in \cite{BMS2} is simply given by $\tau_{\ge 2n-1} \TC(R).$ 
\end{remark}{}

Having discussed these basic properties, we now focus on the behavior of these invariants for the pro-system of truncated polynomial rings.

\begin{proposition}\label{ven}
    Let $R$ be a perfect ring. Then $\varprojlim \mathrm{dR}(R[t]/t^n) \simeq R[[t]]^{(p)} \oplus R[[t]]^{(p)}[-1]$ as an $R[[t]]^{(p)}$-module.
\end{proposition}{}

\begin{proof}
   Let $n = p^k.$ Then $\LL_{R[t]/t^n} \simeq (t^n/t^{2n})[1] \oplus R[t]/t^n$ as an $R[t]/t^n$-module. Let $I_n:= (t^n/t^{2n}),$ which is free of rank $1$ as an $R[t]/t^n$-module. For $s \ge 0,$ one has $$\wedge^s \LL_{R[t]/t^n}[-s] \simeq \Gamma^s(I_n) \oplus \Gamma^{s-1}(I_n)[-1].$$ Now, we note that since $R[t]/t^n$ is liftable to $\Z/p^2 \Z$, along with a lift of the Frobenius, the conjugate filtration on $\mathrm{dR}(R[t]/t^n)$ splits \cite[Prop.~3.17]{Bha12}; this gives
$$\mathrm{dR}(R[t]/t^n) \simeq \Gamma^* (I_n) \oplus \Gamma^* (I_n)[-1].$$ Finally, note that the natural map $R[t]/t^{p^{k+1}} \to R[t]/t^{p^{k}}$ induces the zero map $I_{p^{k+1}} \to I_{p^k}.$ This shows that the $\mathbf N$-indexed objects $\mathrm{dR}(R[t]/t^n)$ and $(R[t]/t^n)^{(p)} \oplus (R[t]/t^n)^{(p)}[-1]$ are isomorphic as pro-objects. This yields the desired claim.
\end{proof}{}

\begin{proposition}\label{pack}
    Let $R$ be a perfect ring. Then $Z_1 \LL \Omega^i_{R[t]/t^n} \simeq 0$ as a pro-object for $i \ge 2.$
\end{proposition}{}

\begin{proof}
    First we prove that the natural map $\mathrm{dR}(R[t]/t^n) \to \widehat{\mathrm{dR}}(R[t]/t^n)$ is a pro-isomorphism. To this end, note that the pro-object $ \widehat{\mathrm{dR}}(R[t]/t^n)$ admits a complete descending filtration (induced by the Hodge filtration) $\mathrm{Fil}^*_{\mathrm{Hodge}}\widehat{\mathrm{dR}}(R[t]/t^n)$ whose graded pieces are described as 
$$ \mathrm{gr}^0 \widehat{\mathrm{dR}}(R[t]/t^n) \simeq R[t]/t^n,\,\, \mathrm{gr}^1 \widehat{\mathrm{dR}}(R[t]/t^n) \simeq \Omega^1_{R[t]/t^n}[-1],\,\, \mathrm{gr}^i\widehat{\mathrm{dR}}(R[t]/t^n) \simeq 0\, \text{for}\, i >1$$ as pro-objects. This implies that the pro-object $\widehat{\mathrm{dR}}(R[t]/t^n)$ is pro-isomorphic to $\Omega^*_{R[t]/t^n},$ where the latter denotes the classical de Rham complex. Now let $n= p^k.$ Then $\Omega^*_{R[t]/t^n},$ as an $(R[t]/ t^n)^{(p)}$-module is naturally isomorphic to $$\Omega^*_{R[t]} \otimes_{R[t]^{(p)}} (R[t]/ t^n)^{(p)}.$$ Considering that $R[t]$ lifts to $\Z/p^2 \Z$ along with a lift of the Frobenius, we see that the pro-object $\widehat{\mathrm{dR}}(R[t]/t^n) \simeq(R[t]/t^n)^{(p)} \oplus (R[t]/t^n)^{(p)}[-1] \simeq \mathrm{dR}(R[t]/t^n),$ where the latter isomorphism follows from the proof of \cref{ven}.

Since for $i \ge 2,$ we have $\widehat{\mathrm{dR}}(R[t]/t^n) \simeq \widehat{\mathrm{dR}}(R[t]/t^n)/ \Fil^i_{\mathrm{Hodge}},$ and the natural map $\mathrm{dR}(R[t]/t^n) \to \widehat{\mathrm{dR}}(R[t]/t^n)$ is a pro-isomorphism, it now follows that the natural map $$\mathrm{dR}(R[t]/t^n) \to \mathrm{dR}(R[t]/t^n)/ \Fil^i_{\mathrm{Hodge}}$$ is a pro-isomorphism. Further, note that the natural map $$\mathrm{Fil}^j_{\mathrm{conj}} \mathrm{dR}(R[t]/t^n) \to \mathrm{dR}(R[t]/t^n)$$ is a pro-isomorphism for $j \ge 1.$ Therefore, the natural map $$\mathrm{Fil}^i_{\mathrm{conj}} \mathrm{dR}(R[t]/t^n) \to \mathrm{dR}(R[t]/t^n)/ \Fil^i_{\mathrm{Hodge}}$$ is a pro-isomorphism for $i \ge 2.$ Thus the fiber, which is naturally isomorphic to $Z_1\LL \Omega^i_{R[t]/t^n}[-i],$ is pro-zero for $i \ge 2$.
\end{proof}{}

\begin{proposition}
    Let $R$ be a perfect ring. Then $\varprojlim \Z_p(i) (R[t]/t^n) = 0$ for $i>1.$
\end{proposition}{}

\begin{proof}
    By derived $p$-completeness, it is enough to prove that  $\varprojlim \F_p(i) (R[t]/t^n) = 0$ for $i>1.$ Now the fiber sequence (see \cref{cartierbloch}) $$ \F_p(i)(R[t]/t^n)[i] \to Z_1 \LL \Omega^i_{R[t]/t^n} \xrightarrow[]{\mathrm{can} - C} \LL\Omega^i_{R[t]/t^n}$$ yields the desired vanishing since $Z_1 \LL \Omega^i_{R[t]/t^n}$ and $\LL\Omega^i_{R[t]/t^n}$ are both pro-zero for $i>1.$
\end{proof}{}

Now we focus our attention to $\varprojlim \Z_p(i) (R[t]/t^n),$ where $R$ is a quasiregular semiperfect algebra. For this purpose, it will be convenient to work with Nygaard filtration on derived crystalline cohomology.

\begin{lemma}
     Let $R$ be a quasiregular semiperfect ring. Then $$ \varprojlim \Fil^{i}_{\conj}\mathrm{dR} (R[t]/t^n)\in D_{[-1,0]} (\Z_p).$$
 
\end{lemma}{}

\begin{proof}
    Note that $\mathrm{dR} (R[t]/t^n) \simeq \mathrm{dR}(R) \otimes_{\F_p} \mathrm{dR} (\F_p [t]/t^n).$ By the proof of \cref{ven}, $\Fil^0_{\conj} \mathrm{dR}(\F_p [t]/t^n) \simeq  (\F_p [t]/t^n)^{(p)}$ and $\Fil^i_{\conj} \mathrm{dR}(\F_p [t]/t^n) \simeq  (\F_p [t]/t^n)^{(p)} \oplus (\F_p [t]/t^n)^{(p)}[-1]$ as $n$-indexed pro-objects for each $i \ge 1$. For fixed $i,n$, we have
$$\Fil^i_{\conj}\mathrm{dR}(R[t]/t^n) \simeq \colim_{u+v \le i} \Fil^u_{\conj} \mathrm{dR}(R) \otimes_{\F_p}\Fil^v_{\conj}   \mathrm{dR}(\F_p [t]/t^n).$$
The above formula gives a natural map
\begin{equation}\label{kt}
    \left( \mathrm{Fil}^*_{\conj} \dR(R) \oplus \mathrm{Fil}^{*-1}_{\conj} \dR(R) [-1] \right) \otimes_{\F_p} (\F_p[t] /t^n)^{(p)} \to \Fil^*_{\conj}\mathrm{dR}(R[t]/t^n)
\end{equation}
To prove that this induces an isomorphism in the category of pro-objects, it is enough to prove that the graded pieces are pro-isomorphic. To this end, let $I:= \mathrm{Ker}(R^\flat \to R);$ then we have $\L_{R/ \F_p} \simeq I/I^2 [1].$ One computes that in the pro-category, we have $$\L_{(R[t]/t^n)/\F_p} \simeq (I/I^2 \otimes_R R[t]/t^n)[1] \oplus R[t]/t^n.$$ Computing wedge powers, we see that \cref{kt} is indeed an isomorphism. Since $\wedge^i \LL_{R/\F_p}[-i] \simeq \Gamma ^i_R (I/I^2)$ is discrete, $\Fil^*_{\conj} \dR(R)$ is also discrete. The maps induced on $\pi_{-1}$ on the left hand side of \cref{kt} are surjections and therefore the proposition follows. 
\end{proof}{}

\begin{lemma}\label{ktt}
    Let $R$ be a quasiregular semiperfect ring. Then $$ \varprojlim\Fil_{\mathrm{Nyg}}^{i} \widehat{R\Gamma}_{\mathrm{crys}}(R[t]/t^n) \in D_{[-1,0]} (\Z_p).$$
\end{lemma}{}

\begin{proof}
    The left hand side is equipped with a complete descending filtration $\varprojlim\Fil_{\mathrm{Nyg}}^{i+*} \widehat{R\Gamma}_{\mathrm{crys}}(R[t]/t^n)$, where the graded pieces are computed by $\varprojlim \Fil^{i+*}_{\conj}\mathrm{dR} (R[t]/t^n).$ Therefore, the claim follows from the above lemma.
\end{proof}{}

\begin{lemma}\label{fight}
    Let $R$ be a quasiregular semiperfect ring. Then $$\varprojlim \widehat{\dR}(R[t]/t^n) \simeq \varprojlim \left(\widehat{\dR}(R) \otimes_{\F_p} (\F_p[t] /t^n)^{(p)} \oplus \widehat{\dR}(R) \otimes_{\F_p} (\F_p[t] /t^n)^{(p)}[-1]\right). $$
\end{lemma}{}

\begin{proof}
It follows that $\widehat{\dR}(R[t]/t^n)$ may be computed by completing $\widehat{\dR}(R) \otimes_{\F_p} \widehat{\dR} (\F_p [t]/t^n)$ with respect to the Day convolution filtration induced from the Hodge filtration $\Fil^*_{\Hodge}\widehat{\dR}(R)$ and $\Fil^*_{\Hodge} \widehat{\dR}(\mathbf F_p[t]/t^n).$ However, as an $\mathbf N$-indexed pro-object, $\Fil^i_{\Hodge} \widehat{\dR}(\mathbf{F}_p[t]/t^n) = 0$ for $i \ge 2$ (see \cref{pack}); therefore, one may ignore the completion step in order to compute the inverse limit, i.e., $\varprojlim \widehat{\dR}(R[t]/t^n) \simeq \varprojlim \left(\widehat{\dR}(R) \otimes_{\F_p} \widehat{\dR} (\F_p [t]/t^n)\right).$ This yields the desired statement.
\end{proof}{}

\begin{proposition}\label{LY}
    Let $R$ be a quasiregular semiperfect ring. Then $$ \varprojlim \Z_p(i)(R[t]/t^n) \in D_{[-1,0]} (\Z_p).$$
\end{proposition}{}

\begin{proof}
When $i=0,$ one may check the claim by reducing modulo $p$ and using the Artin--Schreier sequence. For $i=1,$ we argue as follows: note that for any quasisyntomic $\F_p$-algebra $S,$ we have $\Z_p(1) (S)[1]\simeq R\Gamma_{\mathrm{\acute{e}t}} (S, \mathbb G_m)^{\wedge_p}.$ This implies that we have a fiber sequence $$ \prod_{(r,p)=1} W(S) \to \varprojlim_n \Z_p(1) (R[t]/t^n)[1] \to \Z_p(1) (R)[1].$$Since $R$ is quasiregular semiperfect, $\Z_p(1) (R)$ is discrete, which gives the claim.

Let us now suppose that $i>1.$
  Since we have a fiber sequence of derived $p$-complete objects
$$ \varprojlim \Z_p(i) (R[t]/t^n) \to \varprojlim \Fil_{\mathrm{Nyg}}^{i} \widehat{R\Gamma}_{\mathrm{crys}}(R[t]/t^n) \xrightarrow{\varphi_i - \iota}  \varprojlim  \widehat{R\Gamma}_{\mathrm{crys}}(R[t]/t^n),$$ it would be enough to prove that the map 
${\varphi_i - \iota} \colon \varprojlim \Fil_{\mathrm{Nyg}}^{i} \widehat{R\Gamma}_{\mathrm{crys}}(R[t]/t^n)/p \to \varprojlim  \widehat{R\Gamma}_{\mathrm{crys}}(R[t]/t^n)/p$ induces a surjection on $\pi_{-1}.$ Note that the composition
$$  \Fil_{\mathrm{Nyg}}^{i+1} \widehat{R\Gamma}_{\mathrm{crys}}(R[t]/t^n) \to  \Fil_{\mathrm{Nyg}}^{i} \widehat{R\Gamma}_{\mathrm{crys}}(R[t]/t^n) \xrightarrow{\varphi_i - \iota}  \widehat{R\Gamma}_{\mathrm{crys}}(R[t]/t^n)/p \simeq  \widehat{\dR} (R[t]/t^n)$$ is homotopic to the canonical map $\iota$ and the latter factors as 
\begin{equation}\label{wh}
   \Fil_{\mathrm{Nyg}}^{i+1} \widehat{R\Gamma}_{\mathrm{crys}}(R[t]/t^n) \to \Fil^{i+1}_{\mathrm{Hodge}} \widehat{\dR}(R[t]/t^n) \to  \widehat{\dR}(R[t]/t^n).  
\end{equation}{}
Since we have a fiber sequence
$$ \Fil_{\mathrm{Nyg}}^{i} \widehat{R\Gamma}_{\mathrm{crys}}(R[t]/t^n)  \xrightarrow{p} \Fil_{\mathrm{Nyg}}^{i+1} \widehat{R\Gamma}_{\mathrm{crys}}(R[t]/t^n)  \to \Fil^{i+1}_{\mathrm{Hodge}} \widehat{\dR}(R[t]/t^n),$$it follows from \cref{ktt} that the map $\varprojlim \Fil_{\mathrm{Nyg}}^{i+1} \widehat{R\Gamma}_{\mathrm{crys}}(R[t]/t^n) \to \varprojlim \Fil^{i+1}_{\mathrm{Hodge}} \widehat{\dR}(R[t]/t^n)$ is a surjection on $\pi_{-1}.$ Therefore, the image under $\pi_{-1}$ of the composite map $\varprojlim \Fil_{\mathrm{Nyg}}^{i+1} \widehat{R\Gamma}_{\mathrm{crys}}(R[t]/t^n)  \to \varprojlim \widehat{\dR}(R[t]/t^n)$  coming from \cref{wh} is the same as image of $\pi_{-1}$ induced by the map $$\alpha\colon \varprojlim \Fil^{i+1}_{\mathrm{Hodge}} \widehat{\dR}(R[t]/t^n) \to \varprojlim  \widehat{\dR}(R[t]/t^n).$$

On the other hand, note that the composition $$\Fil^{i-1}_{\mathrm{Nyg}} \widehat{R\Gamma}_{\mathrm{crys}}(R[t]/t^n) \xrightarrow{p} \Fil_{\mathrm{Nyg}}^{i} \widehat{R\Gamma}_{\mathrm{crys}}(R[t]/t^n) \xrightarrow{\varphi_i - \iota} \widehat{R\Gamma}_{\mathrm{crys}}(R[t]/t^n)/p$$ is homotopic to $\varphi_{i-1}.$ Furthermore, since $i>0,$ the map $\varphi_{i-1}\colon\Fil^{i-1}_{\mathrm{Nyg}} \widehat{R\Gamma}_{\mathrm{crys}}(R[t]/t^n) \to \widehat{R\Gamma}_{\mathrm{crys}}(R[t]/t^n)/p $ factors as $$\Fil^{i-1}_{\mathrm{Nyg}} \widehat{R\Gamma}_{\mathrm{crys}}(R[t]/t^n) \to \gr^{i-1}_{\mathrm{Nyg}} \widehat{R\Gamma}_{\mathrm{crys}}(R[t]/t^n) \to \widehat{R\Gamma}_{\mathrm{crys}}(R[t]/t^n)/p .$$Passing to inverse limits over $n$, we see that the map $\Fil^{i-1}_{\mathrm{Nyg}} \widehat{R\Gamma}_{\mathrm{crys}}(R[t]/t^n) \to \gr^{i-1}_{\mathrm{Nyg}} \widehat{R\Gamma}_{\mathrm{crys}}(R[t]/t^n)$ induces surjection on $\pi_{-1}.$ Therefore, the image of the map induced on $\pi_{-1}$ by the composite map $\varprojlim \Fil^{i-1}_{\mathrm{Nyg}} \widehat{R\Gamma}_{\mathrm{crys}}(R[t]/t^n) \to \varprojlim \widehat{R\Gamma}_{\mathrm{crys}}(R[t]/t^n)/p$ is the same as the image of the map induced on $\pi_{-1}$ by $\varprojlim \gr^{i-1}_{\mathrm{Nyg}} \widehat{R\Gamma}_{\mathrm{crys}}(R[t]/t^n) \to \varprojlim \widehat{R\Gamma}_{\mathrm{crys}}(R[t]/t^n)/p.$ The latter map identifies with the map $$\beta\colon\varprojlim \Fil^{i-1}_{\conj} \mathrm{dR} (R[t]/t^n) \to \varprojlim \widehat{\dR} (R[t]/t^n).$$

It would be enough to prove that image of $\pi_{-1}(\alpha)$ and $\pi_{-1}(\beta)$ generates $\pi_{-1} (\varprojlim \widehat{\dR}(R[t]/t^n))$ under addition. Note that there is a natural map $$ \varprojlim \Fil^{i}_{\Hodge} \widehat{\dR}(R) \otimes_{\F_p} \Fil^1_{\Hodge}\widehat{\dR}(\F_p[t]/t^n) \to \varprojlim \Fil^{i+1}_{\mathrm{Hodge}} \widehat{\dR}(R[t]/t^n).$$Composing with $\alpha,$ we get a map $ \varprojlim \Fil^{i}_{\Hodge} \widehat{\dR}(R) \otimes_{\F_p} \Fil^1_{\Hodge}\widehat{\dR}(\F_p[t]/t^n) \to \varprojlim  \widehat{\dR}(R[t]/t^n).$ Note that by \cref{fight}, we have an isomorphism $\pi_{-1} (\varprojlim \widehat{\dR}(R[t]/t^n)) \simeq \widehat{\dR}(R)[[s]].$ It follows that under the latter isomorphism, the image of $\pi_{-1}(\alpha)$ contains all elements of the form $\sum_i x_i s^i$, where $x_i \in \Fil^i_{\Hodge} \widehat{\dR}(R)$ and image of $\pi_{-1}(\beta)$ contains all elements of the form $\sum_i y_i s^i$, where $y_i \in \Fil^{i-2}_{\conj}{\dR}(R)$ (see \cref{kt}). For $i >1,$ we have $$\Fil^{i-2}_{\conj}{\dR}(R) + \Fil^i_{\Hodge} \widehat{\dR}(R) = \widehat{\dR}(R),$$ which finishes the proof. 
\end{proof}{}

\begin{lemma}\label{filtrationlemma}
    Let us suppose that a spectrum $S$ admits a descending complete and exhaustive $\Z$-indexed filtration $\Fil^* S$ such that the graded pieces $\gr^n S \in \mathrm{Sp}_{[2n,2n-1]}.$ Then there is a natural isomorphism $\Fil^n S \simeq \tau_{\ge 2n-1} S.$
\end{lemma}{}

\begin{proof}
 Let us fix an integer $n.$ Let us choose another integer $j \ge 2n-1.$  By the description of graded pieces, it follows that $\pi_j (\Fil^n S) = \pi_j (\Fil^{n-1}S) = \ldots.$ Therefore, $\pi_j (\Fil^n S) \simeq \pi_j (S),$ since the filtration is exhaustive. By completeness of the filtration, we have $\Fil^n S \simeq \varprojlim _{k \in \mathbf N} \Fil^n S/ \Fil^{n+k}S$ By the description of the graded pieces, it follows that if $j < 2n-1,$ then $\pi_j (\Fil^n S/ \Fil^{n+k}) \simeq 0.$ Moreover, using the fiber sequence 
$$\gr^{n+k} S \to \Fil^n S/ \Fil^{n+k+1} S \to \Fil^n S/ \Fil^{n+k} S ,$$ we see that the maps $\pi_{2n-1} (\Fil^n S/ \Fil^{n+k+1}) \to \pi_{2n-1} (\Fil^n S/ \Fil^{n+k})$ are surjections for $k \ge 1.$ By using Milnor sequences, it follows that $\pi_j (\Fil^n S) = 0$ for $j < 2n-1.$ Since we have a natural map $\Fil^n S[-2n+1] \to S[-2n+1],$ and $\Fil^n S[-2n+1]$ is connective, we obtain a map $\Fil^n S \to \tau_{\ge 2n-1} S.$ Since we know that this map induces isomorphism on all homotopy groups, we obtain the desired claim.
\end{proof}{}

Now we can summarize the observations in this section in the following manner: let $S$ be a quasisyntomic $\F_p$-algebra. Let us define $$\Fil^* \varprojlim_k \TC (S[t]/t^k)\coloneqq \varprojlim_k \mathrm{Fil}_{\mathrm{BMS}}^* \TC (S[t]/t^k).$$
It follows that the graded pieces of this filtration are computed as $$\mathrm{gr}^n \varprojlim_k \TC (S[t]/t^k) \simeq \varprojlim_{k} \Z_p(n) (S[t]/t^k)[2n].$$ It also follows that $\Fil^* \varprojlim_k \TC (S[t]/t^k)$ is a complete exhaustive filtration and the functor determined by $S \mapsto \Fil^* \varprojlim_k \TC (S[t]/t^k)$ is a quasisyntomic sheaf of spectra. The proposition below gives a concrete description of this filtration for quasiregular semiperfect algebras.

\begin{proposition}\label{prop1}
    Let $R$ be a quasiregular semiperfect algebra. Then $$\Fil^n \varprojlim_k \TC (R[t]/t^k) \simeq \tau_{\ge 2n-1}  \varprojlim_k \TC (R[t]/t^k).$$
\end{proposition}{}

\begin{proof}
    Follows from the above description of the graded pieces along with \cref{LY} and \cref{filtrationlemma}. 
\end{proof}{}

\begin{remark}
    Let us point out that certain computations of topological cyclic homology of $R[t]/t^k$ where $R$ is a perfect(oid) ring appeared in \cite{Yuri} and \cite{Noah1} (\textit{cf.}~\cite[Thm.~10.4]{akhill}).
\end{remark}{}

\section{Proof of the main result}\label{finalsec} In this section, we will enhance Hesselholt's isomorphism \cref{hesselholt} with the motivic filtrations studied above. We recall some notations first. Let $S$ be a quasisyntomic $\F_p$-algebra. Let $\Fil^* \TR(S)$ be the filtration constructed before \cref{cor1}. Let $\Fil^*  \TC (S[t]/t^k)$ and $\Fil^* \TC(S)$ be the motivic filtrations as constructed by Bhatt--Morrow--Scholze.

\begin{proposition}\label{mainthm1}
    Let $S$ be a quasisyntomic $\F_p$-algebra. Then  we have a natural isomorphism $$ (\Fil^{*-1} \TR(S))[1] \simeq \varprojlim_k \mathrm{fib} \left(\Fil^* \TC(S[t]/t^k) \to \Fil^* \TC(S)\right).$$
\end{proposition}{}

\begin{proof}
    Let $R$ be a quasiregular semiperfect algebra. Using Hesselholt's result \cref{hesselholt}, we obtain a natural fiber sequence $$ \varprojlim_k \TC (R[t]/t^k) \to \TC(R) \to \TR(R)[2].$$ 
By \cref{cor3} and \cref{BMSsimple}, we obtain a fiber sequence
$$\tau_{\ge 2n-1}  \varprojlim_k \TC (R[t]/t^k) \to \tau_{\ge 2n-1} \TC(R) \to 
(\tau_{\ge 2n-2}\TR(R))[2].$$ Using \cref{cor3} and \cref{prop1}, this implies that we have a fiber sequence
$$\Fil^n \varprojlim_k \TC (R[t]/t^k) \to \Fil^n \TC(R) \to \Fil^{n-1} \TR(R) [2].$$ Applying quasisyntomic descent produces a natural isomorphism $$(\Fil^{n-1} \TR(S))[1] \simeq \varprojlim_k \mathrm{fib} \left(\Fil^n \TC(S[t]/t^k) \to \Fil^n \TC(S)\right)$$ for any quasisyntomic $\mathbb{F}_p$-algebra $S$. This finishes the proof.
\end{proof}{}

\begin{proposition}\label{prop98}
Let $S$ be a quasisyntomic $\F_p$-algebra. Then we have a natural isomorphism
$$\prod_{(u,p)=1}\L W\Omega^{n-1}_S \simeq \mathrm{fib} \left(\varprojlim_k \Z_p(n)(S[t]/t^k)[n] \to \Z_p(n)(S)[n]\right).$$
\end{proposition}

    \begin{proof}
        By passing to the graded pieces in the filtered isomorphism in \cref{mainthm1}, we obtain a natural isomorphism $$\prod_{(u,p)=1}\L W\Omega^{n-1}_S[n-1][1] \simeq \mathrm{fib} \left(\varprojlim_k \Z_p(n)(S[t]/t^k)[2n] \to \Z_p(n)(S)[2n]\right),$$ which gives the desired result.\end{proof}{}

\begin{construction}[Frobenius and Verschiebung]\label{constrans}
Let $S$ be a quasisyntomic $\mathbf F_p$-algebra. Let $$C (\mathbf{Z}_p(n)[n]_S) \coloneqq \mathrm{fib} \left(\varprojlim_k \Z_p(n)(S[t]/t^k)[n] \to \Z_p(n)(S)[n]\right),$$ which we regard as curves on $\mathbf Z_p(n)[n].$ Let $m \ge 0$ be an integer. The assignment $t \mapsto t^m$ determines an endomorphism  $V_m\colon C (\mathbf{Z}_p(n)[n]_S) \to C (\mathbf{Z}_p(n)[n]_S),$ which we call the $m$-th \emph{Verschiebung}. We will now construct the $m$-th Frobenius maps. Let us first assume that $S$ is quasiregular semiperfect. We will construct a ``transfer endomorphism" \begin{equation}\label{transfer1}
    \Phi_m\colon \varprojlim_k \Z_p(n)(S[t]/t^k)  \to \varprojlim_k \Z_p(n)(S[t]/t^k)
\end{equation}{}To do so, one notes that there are transfer maps $ \TC (S[t]/t^{km}) \to \TC (S[t]/t^k)$ induced by the map $S[t]/t^{k} \to S[t]/t^{km}$ determined by $t \mapsto t^m.$ This induces a map $$\Tilde{\Phi}_m\colon \varprojlim_k \TC(S[t]/t^k) \to \varprojlim_k \TC(S[t]/t^k).$$ Now, using the assumption that $S$ is quasiregular semiperfect, \cref{prop1}, and passing to graded pieces produces the desired transfer map \cref{transfer1} on the $p$-adic Tate twists. By quasisyntomic descent, for any quasisyntomic $\mathbf F_p$-algebra, we obtain a ``transfer map" \begin{equation}\label{transfer}
    \Phi_m\colon \varprojlim_k \Z_p(n)(S[t]/t^k)  \to \varprojlim_k \Z_p(n)(S[t]/t^k).
\end{equation} The latter induces an endomorphism $$F_m \colon C (\mathbf{Z}_p(n)[n]_S) \to C (\mathbf{Z}_p(n)[n]_S),$$ that we call the $m$-th \emph{Frobenius}.
\end{construction}

\begin{remark}
    Note that the Frobenius and Verschiebung operators on $\TR(S)$ induce the operators $F_m$ and $V_m$ on the left hand side of \cref{prop98} as well. Using \cite[Rmk.~2.4.6]{Jonas}, it follows that \cref{prop98} is compatible with the $F_m$ and $V_m$ defined on both sides.
\end{remark}{}

\begin{construction}[$p$-typicalization]\label{ptypica}
Let $S$ be a quasisyntomic $\mathbf F_p$-algebra. Using \cref{constrans}, we obtain natural maps $\eta_m\colon \mathrm{fib} (F_m) \to C (\mathbf{Z}_p(n)[n]_S),$ which maybe viewed as an object of $D(\mathbf Z_p)_{/C (\mathbf{Z}_p(n)[n]_S)}.$ We define $$\mathbb{D}(\Z_p(n)[n]_S) \coloneqq \prod_{(m,p)=1, m>1}\eta_m \in D(\mathbf Z_p)_{/C (\mathbf{Z}_p(n)[n]_S)},$$ where the product is taken in $D(\mathbf Z_p)_{/C (\mathbf{Z}_p(n)[n]_S)}$. Naively, one may think of $\mathbb{D}(\Z_p(n)[n]_S)$ as $``\bigcap_{(m,p)=1, m>1} \mathrm{fib}(F_m)",$ where the latter should be suitably interpreted as above. By analogy with the classical situation, we will call $\mathbb{D}(\Z_p(n)[n]_S)$ the \emph{$p$-typical curves} on $\mathbf Z_p(n)[n]$ over $S.$ Note that $\mathbb{D}(\Z_p(n)[n]_S)$ is naturally equipped with the operators $F\coloneqq F_p$ and $V\coloneqq V_p$.
\end{construction}{}

\begin{corollary}Let $S$ be a quasisyntomic $\F_p$-algebra. Then we have a natural isomorphism 
$$ \LL W \Omega_S^{n-1} \simeq \mathbb{D}(\Z_p(n)[n]_S),$$ which is compatible with the $F$ and $V$ defined on both sides.
\end{corollary}{}

\begin{proof}
This follows from \cref{prop98}, the previous discussion and the description of the Frobenius on $\TR(S)$ following \cite[Prop.~3.3.1]{Hess}.   
\end{proof}{}
\bibliographystyle{amsalpha}
\bibliography{main}
\end{document}